\documentclass[12pt,a4paper]{article}
\usepackage[utf8]{inputenc}
\usepackage{amsmath,amssymb,amsfonts}
\usepackage{amsthm}
\usepackage{paralist}
\usepackage{mathtools}
\usepackage{color}
\usepackage{multicol}
\usepackage{authblk}
\usepackage{subcaption} 
\usepackage[font=scriptsize,labelfont=bf]{caption}
\usepackage{dsfont}
\usepackage{float}
\usepackage{natbib}

\newtheoremstyle{style}
  {0.5cm}                 %Space above    
  {0.5cm}                 %Space below    
  {\itshape}                         %Body font: original {\normalfont}    
  {}                         %Indent amount (empty = no indent,%\parindent = paraindent)    
  {\normalfont\bfseries}  %Thm head font original 
  {\normalfont {\bf : }} 
  {\newline}
  {}
  
\theoremstyle{style}

\newtheorem{theorem}{Theorem}[section]
\newtheorem{definition}{Definition}[section]
\newtheorem{lemma}{Lemma}[section]

\newtheorem{ass}{Assumption}[section]
\newtheorem{remark}{Remark}[section]

\newcommand{\R}{\mathbb{R}}
\newcommand{\N}{\mathbb{N}}

\begin{document}

\title{A Nonlocal Spatial Ramsey Model with Endogenous Productivity Growth on Unbounded Spatial Domains}
\author[*]{L. Frerick}
\author[**]{G. M{\"u}ller-F{\"u}rstenberger}
\author[*]{E. W. Sachs}
\author[*]{L. Somorowsky}

\affil[*]{Department of Mathematics, University of Trier, Germany}
\affil[**]{Department of Economics, University of Trier, Germany}

\maketitle

\section{Introduction}

%\cite{aldashev}\\
%\cite{boucekkine,boucekkine13}\\
%\cite{brito01,brito04,brito12} \\
%\cite{camacho}\\
%\cite{krugman1}\\
%
%`The heterogeneity matters'\\
%agglomeration vs dispersive forces\\
%new economic geography higher concentrated on agglomeration\\
%Brito model: high concentration on dispersive forces, hardly any agglomeration\\
%however, neither diffusive nor agglomerative effects should not be denied\\
% $\rightarrow$ our model combines both forces and balances them\\
%closed economy, no interaction with surrounding
%
%
%
%{\bf combined local-nonlocal diffusion, no boundary conditions}\\

In the recent literature dealing with spatial extensions of the continuous Ramsey model (cf. \citealp{brito01,brito04,brito12}, \citealp{boucekkine,boucekkine13}, or  \citealp{camacho}), the capital accumulation process via time and space is  modeled as a linear parabolic partial differential equation. The process of capital movement through space is described by a Laplace operator. This operator makes use of information on the capital distribution in the respective point and in an arbitrary small neighborhood. In that way, it describes spreading phenomena, where molecules always need physical contact to the direct surrounding to move from one location to another. Considering labor force, the agents endowed with labor may move from one location to another, without undertaking work in every single location they pass through on their way. The same behavior is observable for the dispersion of capital. Investments arise only on some separate locations, they do of course affect the surrounding, but do not spread evenly from one point to another. Capital, as well as labor force, can literally `jump' through space. Hence, a local diffusion operator may not capture the true behavior of the natural dispersion effect of the production factors in an economy.\\

An operator which is more appropriate in this setting is a {\em nonlocal diffusion operator}. The nonlocal diffusion equation, we will consider here, arises naturally from a probabilistic process in which capital moves randomly in space, subject to a probability that allows long jumps. However, the model is not stochastic, but deterministic as we consider an economy with a central planner, who observes any spatial consumption distribution in all points of time, and can determine the capital distribution according to the capital accumulation process.\\
Moreover whereas the Laplace operator insinuates an infinite adjustment speed of the molecules, the nonlocal diffusion operator decelerates this process, which fits real world observations better. In that way, we are able to conserve heterogeneities over a much longer time horizon and to even preserve discontinuities in initial capital or productivity distributions.\\
To our knowledge, we are the first who introduce such {\em nonlocal diffusion effects} in the spatial Ramsey model. In our version of this economic growth model, capital mobility in a location does not only depend on the respective one but also on `far away' locations. \\

A groundbreaking innovation of the Ramsey model is the endogenous saving rate, which means that the optimal saving rate, that maximizes the welfare of the economy, is determined via the households' lifetime maximization intention during the optimization process within the model itself. In that point, the Ramsey model differs from many other neoclassical growth models. Economic growth is also driven by technological progress, or the increase of productivity, which can both be modeled by so called {\em spillover effects}. In the common (local) Ramsey model, this productivity growth is assumed to be growing at a constant rate $A$ (cf. \citealp{brito01,brito04,brito12}, \cite{boucekkine13}). In our opinion, this exogenously pre-defined productivity growth rate sets the endogenous character of the Ramsey model aside. We introduce a new, nonlocal productivity operator $P$, that aims to endogenize the process of productivity growth, and in that way, preserves the self-contained character of the Ramsey model. We assume that there is a correlation between the development (meaning an increase) of productivity and the state of the system, namely the capital stock in a surrounding of a respective location. Moreover, we assume that productivity naturally increases over time. We model the productivity growth as integral term as well. \\

%{\bf Grunds\"atzlich alles etwas mathematischer?}

\section{The Model}

The main aspect in the Ramsey model is the (competitive) equilibrium growth. 

\begin{definition}[Competitive Equilibrium]
A {\rm competitive equilibrium} consists of paths of consumption, capital stock, wage rates and rental rates of capital, $\{C_t,K_{t+1},\omega_t,R_t\}_{t=0}^T$, such that the representative household maximizes its utility given an initial capital stock $K_0$ and prices $\{\omega_t,R_t\}$ and the path of prices is such that, given the path of capital stock and labor $\{K_t,L_t\}_{t=1}^{T+1}$, all markets clear.
\end{definition}

In the originally space independent model, \cite{ramsey} himself considered an infinite time horizon. This assumption is appropriate from an economic point of view. Although no agent lives forever, this non-terminated time naturally introduces a  sustainability condition. In some discrete models, as for example introduced by \citet[Chapter~6]{acemoglu}, an immortal agent is  explicitly interpreted as a dynasty, where single individuals have the incentive to pass a non-zero capital stock to future generations. Whenever a space dimension is introduced to the Ramsey model, it is necessary to decide whether the spatial domain should be bounded or not. The combination of an infinite time horizon and an unbounded spatial domain holds some difficulties concerning the well-posedness of the spatial model (cf. \citealp[p.3]{boucekkine}). As already pointed out in Chapter \ref{NSRMWEPG}, we circumvent these difficulties by introducing a terminal capital distribution $k_T$ that should not be  undercut. In this way, we mimic an infinite time horizon, but do only have to deal with a finite terminal time. Moreover, we introduce a spatial discounting in the objective function, which is convenient in the setting of a central planner. 
These additional constraints on the state variable and the special structure of the objective function allow us to consider an unbounded spatial domain in the nonlocal spatial Ramsey model. Such infinite space domains are of interest because they can be interpreted as one single and closed economy, where no flows of production factors to, or interactions with any other economies take place. Moreover, due to the spatial discounting, we do not need to define any boundary conditions in order to guarantee well-posedness of the model. Several types of boundary conditions, such as Neumann, Dirichlet, or Cauchy conditions and their economic meaning are for example discussed by \citet[p.14]{brito04}. Here,  it becomes obvious that the choice of the appropriate type of boundary conditions is not an easy task and that it heavily influences the solution of the underlying partial differential equation. \cite{camacho} also consider a finite time horizon and unbounded spatial domain, but they disclaim any spatial discounting. This is the reason why they have to introduce free boundary conditions that enforce the capital distribution to become flat towards infinity, what restricts the set of possible solutions of the partial differential equation too much.\\

The capital accumulation equation, which we consider in the following, is a mixed local-nonlocal diffusion equation, i.e. the diffusion weights $\alpha,\beta$ that we introduce are both positive. We see later that we can indeed choose $\alpha$, which is the weight of the local diffusion term, very small but that we cannot neglect it. Due to the unbounded spatial domain, we do not have to introduce any boundary or volume constraints. Moreover, we do not have to truncate the kernel function in the nonlocal diffusion operator in this setting, but are able to analyze the dynamics of the Ramsey model on the whole, unbounded, and untruncated spatial domain. We fix the finite time horizon $T\in\N$. The spatial domain we consider here is the untruncated $\R^n$, $n\in\N$. Hence, the nonlocal version of the capital accumulation equation of the Ramsey model which we consider here is defined as

\begin{equation}\label{semilineq}
\begin{split}
k_t  -\mathcal{L}(k)+\delta k - \mathcal{P}(k) &= - c\ \mbox{\hspace{2.1cm} on } \R^n \times (0,T),\\[2mm]
k(\cdot,0)&=k_0(\cdot)>0\ \hspace{1.1cm} \mbox{ in }\R^n,
\end{split}
\end{equation}
where the local-nonlocal diffusion operator $\mathcal{L}$ is defined as

\begin{equation}\label{L}
\mathcal{L}(k)(x,t):=\alpha\ \Delta k(x,t)+ \beta\int_{\R^n} (k(y,t)-k(x,t))\Gamma_\varepsilon(x,y)dy,
\end{equation}

for coefficients $\alpha,\beta > 0$ and $\varepsilon>0$. From an application point of view, it is useful to consider the density function of the multivariate normal distribution as kernel function, hence
\begin{equation}
\Gamma_{\varepsilon}(x,y):=\frac{1}{\sqrt{(2\pi\varepsilon^2)^n}}\exp\left(-\frac{1}{2}(x-y)^T\Sigma_\varepsilon^{-1}(x-y)\right),
\end{equation}
for a given covariance matrix $\Sigma_\varepsilon$ with $\det(\Sigma_\varepsilon)=\varepsilon^{2n}$, $\varepsilon>0$. In the following, we assume that the matrix $\Sigma_\varepsilon$ is a diagonal matrix with constant entries,
\begin{equation*}
\Sigma_\varepsilon=\begin{bmatrix}
\varepsilon^2 & &\\
& \ddots & \\
& & \varepsilon^2
\end{bmatrix}\in\R^{n\times n}.
\end{equation*}
This assumption is application driven. We assume that capital can move through space without any barriers, or transition costs, thus capital flows are absolutely free in space. Moreover, the central planner does not prioritize any space direction, but weights them all equally. Hence, the spatial directions in the spatialized Ramsey model are completely uncorrelated and the variances are equal.\\
Given this special form of the covariance matrices, we can rewrite the kernel function as
\begin{equation}\label{gamma}
\Gamma_{\varepsilon}(x,y)=\frac{1}{\sqrt{(2\pi\varepsilon^2)^n}}\exp\left(-\frac{\|x-y\|_2^2}{2\varepsilon^2}\right), 
\end{equation}
where $\|\cdot\|_2$ denotes the Euclidean norm.\\

The nonlocal operator $\mathcal{P}$ on the left-hand side describes the production of the economy and is given as
\begin{equation}\label{P}
\begin{split}
\mathcal{P}(k)(x,t):&=P(k)(x,t) \ p(k(x,t))   \qquad  \mbox{where} \\[3mm]
P(k)(x,t) &= A_0(x) \exp\left(\frac{\int_{\R^n}\phi(k(y,t))\Gamma_{\mu}(x,y)dy}{\int_{\R^n}\phi(k(y,t))\Gamma_{\varepsilon}(x,y) dy+\xi}\ t\right) \ ,
\end{split}
\end{equation}
and $A_0:\R^n\to\R$ denotes the initial productivity distribution over space, $\phi:\R\to\R_+$ is the continuous nominal function, and $p:\R\to\R$ denotes the productivity function. The kernel function $\Gamma_\mu$ is defined analogously to (\ref{gamma}) for a parameter $0<\mu<\varepsilon$. The boundedness of the fraction in the exponential function is an important property, that we will exploit very often in this chapter. We state this property in the next lemma.

\begin{lemma}
Let $\phi:\R\to\R_+$ be continuous with $\phi(x) \leq a_0 + a_1 x$ for all $x \in \R$ and some $a_0, a_1 \in \R$. Let $\xi>0$ and the kernel functions $\Gamma_\mu$ and $\Gamma_\varepsilon$ for parameters $0<\mu \leq \varepsilon$ be defined according to equation (\ref{gamma}). Then the estimate
\begin{equation}\label{estimateGamma}
\frac{\int_{\R^n}\phi(k(y,t))\Gamma_{\mu}(x,y)dy}{\int_{\R^n}\phi(k(y,t))\Gamma_{\varepsilon}(x,y) dy+\xi} \le \left(\frac{\varepsilon}{\mu}\right)^n
\end{equation}
holds for all $x\in \R^n$.
\end{lemma}

\begin{proof}
Without any loss of generality, we choose $x=0$. As $x$ is by definition the expected value of $\Gamma_\nu$, $\nu\in\{\mu,\varepsilon\}$, the proof will be analog for every other $x$, but with translational displaced $\Gamma_\nu$. We define $\Gamma_\nu(0,y)=:\Gamma_\nu(y)$. The inequality (\ref{estimateGamma}) can be rewritten as
\begin{equation*}
\begin{split}
\int_{\R}\phi(k(y,t))\Gamma_{\mu}(y)dy \le \left(\frac{\varepsilon}{\mu}\right)^n\left( \int_{\R}\phi(k(y,t))\Gamma_{\varepsilon}(y) dy + \xi\right).
\end{split}
\end{equation*}
Subtracting the left term, we get
\begin{equation*}
\begin{split}
0 \le \int_{\R}\phi(k(y,t))\left( \left(\frac{\varepsilon}{\mu}\right)^n\Gamma_{\varepsilon}(y)- \Gamma_\mu(y)\right) dy + \left(\frac{\varepsilon}{\mu}\right)^n\xi,
\end{split}
\end{equation*}
which is in particular true whenever 
\[\left(\frac{\varepsilon}{\mu} \right)^n\Gamma_{\varepsilon}(y)- \Gamma_\mu(y) \ge 0 \]
for all $y\in\Omega$, since we assume $\phi$ to be nonnegative. But this inequality follows with the monotonicity of the exponential function. Let therefore $y\in\Omega$ be arbitrary, then we have
\begin{equation*}
%\begin{split}
\left(\frac{\varepsilon}{\mu}\right)^n \Gamma_{\varepsilon}(y)- \Gamma_\mu(y)
%& =  \left(\frac{\varepsilon}{\mu}\right)^n \frac{1}{\sqrt{(2\pi\varepsilon^2)^n}}\exp\left( -\frac{\|y\|_2^2}{2\varepsilon^2}  \right) \\
%& \quad- \frac{1}{\sqrt{(2\pi\mu^2)^n}}\exp\left( -\frac{\|y\|_2^2}{2\mu^2}  \right)\\
  = \frac{1}{\sqrt{(2\pi\mu^2)^n}}\left( \exp\left( -\frac{\|y\|_2^2}{2\varepsilon^2}  \right) - \exp\left( -\frac{\|y\|_2^2}{2\mu^2}  \right)\right)\ge 0
%\end{split}
\end{equation*}
whenever $\mu \leq \varepsilon$,
%\[-\frac{\|y\|_2^2}{2\varepsilon^2} \ge -\frac{\|y\|_2^2}{2\mu^2}, \]
which completes the proof.
\end{proof}

\section{Existence of Solutions of the Capital Accumulation Equation}
The nonlocal capital accumulation equation we defined earlier is a semilinear parabolic partial integro-differential equation. The local diffusion operator allows us to use the common methodology to analyze the PIDE with respect to well posedness in the weak sense. In order to do so, we derive the weak formulation of (\ref{semilineq}) multiplying the equation with a function $v\in H^1(\R^n)$ and integrating over $\R^n$. Integrating by parts, we get

\begin{equation*}
\begin{split}
&\int_{\R^n} k_t(x,\cdot)\ v(x) \ dx\ +\ \int_{\R^n}\ (\alpha\  \nabla_x k(x,\cdot)^T \nabla_x v(x)\   +  \delta  k(x,\cdot)\ v(x)) \ dx \\
&- \beta \int_{\R^n} \int_{\R^n} (k(y,\cdot)-k(x,\cdot))\Gamma_\varepsilon(x,y)dy\ v(x) \  dx\  = \int_{\R^n} \left( \mathcal{P}(k)(x,\cdot)\ -\ c(x)\right) v(x) \ dx ,
\end{split}
\end{equation*}

where the equality has to be understood in distributional sense with respect to $t$.\\

This weak formulation motivates the following definition of a bilinear form.

\begin{definition}
We define the bilinear form  $\ {\bf a}:H^1(\R^n)\times H^1(\R^n)\to\R$ by
\begin{equation}\label{a}
\begin{split}
{\bf a}(u,v):&=\alpha\ \int_{\R^n}\ \nabla_x u^T \nabla_x v\ dx \ + \ \delta \int_{\R^n} u\ v \ dx\\
&- \beta \int_{\R^n} \int_{\R^n} (u(y)-u(x))\Gamma_\varepsilon(x,y)\ dy\ v(x) \ dx.
\end{split}
\end{equation}
\end{definition}

In order to apply a quite abstract result of \citet[p. 513]{dautray}, where the authors prove the existence of a weak solution of a linear PDE, we have to show that the bilinear form ${\bf a}$ is continuous and weakly coercive. Note that this bilinear form is independent of time, since we have chosen $\alpha$ and $\beta$ to be constants. \\ 

\begin{lemma}\label{PropyA}
There exist some constants $c_1,c_3>0$, and $c_2\ge 0$ such that the bilinear form ${\bf a}$ as defined in (\ref{a}) satisfies the following properties for all functions $u,v\in H^1(\R^n)$:
\begin{equation}
\begin{split}
(i)& \mbox{Continuity: }|{\bf a}(u,v)|\le c_1\|u\|_{H^1(\R^n)}\|v\|_{H^1(\R^n)},\\[3mm]
(ii)&\mbox{G\r{a}rding Inequality: } {\bf a}(u,u)+c_2\|u\|_{L^2(\R^n)}^2\ge c_3\|u\|^2_{H^1(\R^n)}.
\end{split}
\end{equation}
\end{lemma}

\begin{proof}\ \\
$(i)$ For the first and second term of the bilinear form defined in (\ref{a}), we have 
\begin{equation*}
\begin{split}
\left|\int_{\R^n} \left( \alpha \ \nabla_x u^T \nabla_x v +\delta\ uv\right) \ dx\right| & \le (\alpha+\delta)\|u\|_{H^1(\R^n)}\|v\|_{H^1(\R^n)}\qquad \forall\  u,v \in H^1(\R^n),
\end{split}
\end{equation*}
using the H\"older inequality two times and the definition of the $H^1(\R^n)$ norm.\\
In order to estimate the nonlocal term, a little more work has to be done. We rewrite the term for $y:=x-z$ and apply the fundamental theorem of calculus. This yields 

\begin{equation*}
\begin{split}
& \left|\int_{\R^n}\int_{\R^n} (u(x-z)-u(x))\frac{1}{\sqrt{(2\pi\varepsilon^2)^n}}\exp\left(-\frac{\|z\|_2^2}{2\varepsilon^2}\right)dz\ v(x)\  dx \right|\\
=&  \left|\int_{\R^n}\int_{\R^n} \int_0^1 \nabla_x u(x-\xi z)^T z d\xi\ \frac{1}{\sqrt{(2\pi\varepsilon^2)^n}}\exp\left(-\frac{\|z\|_2^2}{2\varepsilon^2}\right)dz\ v(x)\  dx\right|\\
%=&  \left| \int_0^1 \int_{\R^n} \int_{\R^n} \nabla_x u(x-\xi z) v(x) \  dx\ z\ \frac{1}{\sqrt{(2\pi\varepsilon^2)^n}}\exp\left(-\frac{\|z\|_2^2}{2\varepsilon^2}\right)dz d\xi \right|\\
%\le &  \left| \int_0^1 \int_{\R^n}\left( \int_{\R^n} |\nabla_x u(x-\xi z)|^2 d x\right)^{1/2}\|v\|_{L^2( \R^n)} \ z\ \frac{1}{\sqrt{(2\pi\varepsilon^2)^n}}\exp\left(-\frac{\|z\|_2^2}{2\varepsilon^2}\right)dz d\xi \right|\\
%=&  \left| \int_0^1 \int_{\R^n}\left( \int_{\R^n} |\nabla_x u(y)|^2\ dy\right)^{1/2}\ \|v\|_{L^2(\R)} \ z\ \frac{1}{\sqrt{(2\pi\varepsilon^2)^n}}\exp\left(-\frac{\|z\|_2^2}{2\varepsilon^2}\right)dz d\xi \right|\\
\le &   \int_{\R^n} \|\nabla_x u\|_{L^2(\R^n)} \|v\|_{L^2(\R^n)}\  |z| \ \frac{1}{\sqrt{(2\pi\varepsilon^2)^n}}\exp\left(-\frac{\|z\|_2^2}{2\varepsilon^2}\right)dz \\[1mm]
%\le &  \|\nabla_x u\|_{L^2(\R^n )}\|v\|_{L^2(\R^n)}\left| \int_{\R^n} \ z \ \frac{1}{(\sqrt{2\pi\varepsilon^2})^n}\exp\left(-\frac{\|z\|_2^2}{2\varepsilon^2}\right)dz \right|\\
\le &  \kappa\|\nabla_x u\|_{L^2(\R^n)}\|v\|_{L^2(\R^n)} \leq \kappa\|u\|_{H^1(\R^n )}\|v\|_{H^1(\R^n)},
\end{split}
\end{equation*}
with a constant $\kappa$,
%\[0\le \kappa:= \int_{\R^n}  |z|\ \frac{1}{\sqrt{(2\pi\varepsilon^2)^n}}\exp\left(-\frac{\|z\|_2^2}{2\varepsilon^2}\right) dz .\]
which is finite, since $\int_{\R^n} |x|\exp(-a\|x\|_2^2)\ dx$ is bounded whenever $a$ is positive.\\
Combining all estimates, the continuity of ${\bf a}$
%\[|{\bf a} (u,v)|\le c_1\|u\|_{H^1(\R^n )}\|v\|_{H^1(\R^n )}\qquad\forall\ u,v\in H^1(\R^n ),\]
with $c_1:= \alpha +\delta+\beta\kappa$, is proven.
\ \\

$(ii)$ To prove the weak coercivity of ${\bf a} $, the procedure is the same as in $(i)$, hence every term is estimated separately. For the first term, we have
\[ \int_{\R^n} (\alpha\ |\nabla_x u|^2 + \delta\ u^2) \ dx= \alpha \|\nabla_x u\|_{L^2(\R^n )}^2 + \delta \|u\|_{L^2(\R^n )}^2 = \alpha\|u\|_{H^1(\R^n )}^2 + (\delta-\alpha) \|u\|^2_{L^2(\R^n ).}\]

The estimate from $(i)$ leads to the following:

\begin{equation*}
\begin{split}
& - \beta \int_{\R^n}\int_{\R^n} (u(x-z)-u(x))\Gamma_\varepsilon(z) dz\ u(x)\ dx
\ge  - \beta \kappa \|\nabla_x u\|_{L^2(\R^n )} \|u\|_{L^2(\R^n )}.
\end{split}
\end{equation*}
Using Young's inequality for an arbitrary $c > 0$, we get
\begin{equation*}
\begin{split}
- \beta\kappa \|\nabla_x u\|_{L^2(\R^n )} \|u\|_{L^2(\R^n )}  \ge - \frac{c}{2}\|\nabla_x u\|_{L^2(\R^n )}^2 - \frac{(\beta\kappa) ^2}{2c}\|u\|_{L^2(\R^n)}^2.
\end{split}
\end{equation*}

Combining both estimates then completes the proof,

\[{\bf a}  (u,u)\ge \alpha \|u\|^2_{H^1(\R^n )}-\left(\alpha+\frac{(\beta\kappa) ^2}{2c} -\delta\right)\|u\|^2_{L^2(\R^n )} - \frac{c}{2}\|\nabla_x u\|_{L^2(\R^n )}^2\]
which yields (ii) by choosing $c$ sufficiently small.
%which is equivalent to
%\[{\bf a} (u,u) + \left(\alpha+\frac{(\beta\kappa) ^2}{2c}-\delta\right)\|u\|_{L^2(\R^n )}^2 \ge \left(\alpha-\frac{c}{2}\right)\|u\|_{H^1(\R^n)}^2 .\]
\end{proof}

\begin{remark}
%a) 
%At the end of the proof of Lemma \ref{PropyA} $(ii)$, we can choose $c>0$ small enough such that 
%\[c_3:=\left(\alpha-\frac{c}{2}\right) >0 \quad \text{ and } \quad c_2:=\left(\alpha+ \frac{(\beta\kappa) ^2}{2c} - \delta\right)\ge 0.\]
%In that case, we get the G\r{a}rding inequality 
%\[{\bf a} (u,u)+c_2\|u\|_{L^2}^2\ge c_3\|u\|^2_{H^1}.\]
%Instead of considering $k\in W(0,T)$ satisfying the capital accumulation equation, we can consider a function 
%\[z(t)=k(t)\exp(-c_2t)\ \in W (0,T),\]
%which has to satisfy the same equation (with a slightly modified right-hand side) and whose corresponding bilinear form is strictly coercive. Hence,  we can interpret the bilinear forms in the proofs below as the one of $z$ and assume the coercivity of ${\bf a}$ without any loss of generality (cf. \citealp[p. 512]{dautray} or \citealp[p.384]{wloka}). \\

Note that at the end of the proof, we need the parameter $\alpha$, which is the weighting parameter of the local diffusion operator, to be positive such that $\alpha-\frac{c}{2}$ is positive. The constant $c$, which comes from Young's inequality, is positive, so we cannot choose $\alpha=0=c$. Hence, at this point it becomes obvious why we need the local diffusion term in the Ramsey model on unbounded spatial domains. Nevertheless, we can choose $c$ to be very small and so are able to minimize the local diffusion effect in the spatial Ramsey model over unbounded spatial domains.
\end{remark}

In order to achieve the result on the existence of weak solutions of linear PDEs by \cite{dautray} also for  the semilinear case we are studying here, we need to state some assumptions on the nonlinearities in our model. 
%Although they may seem quite restrictive, the assumptions are appropriate also from economic point of view.

\begin{ass}\label{AssP}
a) The nonlinear functions in the nonlocal spatial Ramsey model with endogenous productivity growth are assumed to satisfy the following properties: The production function $p:\R\to\R$
\begin{itemize}
\item  is concave and Lipschitz continuous, hence there exists a constant $L_p>0$, such that
\[|p(x)-p(y)|\le L_p |x-y|,\ \forall \ x, y \in \R. \]
\item  is bounded, hence there exists a constant $M_p>0$, such that
\[|p(x)| \le M_p,\ \forall\ x\in \R.\]
\item  satisfies 
\[p(0)=0.\]
\end{itemize}
b) The initial productivity distribution satisfies $A_0\in L^2(\R^n)\cap L^\infty(\R^n)$.\\
c) The nominal function $\phi:\R\to\R_+$ is Lipschitz continuous with Lipschitz constant $L_\phi>0$, hence
\[|\phi(x)-\phi(y)|\le L_\phi |u-v|,\ \forall\ x, y\in\R\]
and satisfies for some $a_0, a_1$
\[ \phi(x) \leq a_0 + a_1 x \quad \mbox{ for all} ~x \in \R. \]
\end{ass}

The following regularity of the productivity-production operator $\mathcal{P}$ will be crucial in the proof of existence of a weak solution of the capital accumulation equation.

\begin{lemma}\label{regularityP}
Let Assumption \ref{AssP} be valid. The operator $\mathcal{P}$ is bounded and Lipschitz continuous in the following sense:
\begin{equation} \label{Pbdd}
\| \mathcal{P}(k)(\cdot,t) \|_{L^2(\R^n)} \leq c_1 \|k\|_{L^2(\R^n)}  \quad \mbox{for all}~~ t \in [0,T]
\end{equation}
and
\begin{equation} \label{Plip}
\| \mathcal{P}(k_1)(\cdot,t)  - \mathcal{P}(k_2)(\cdot,t) \|_{L^2(\R^n)} \leq c_2 \|k_1 - k_2\|_{L^2(\R^n)} \quad \mbox{for all}~~ t \in [0,T]
\end{equation}
hold for all $k, k_1, k_2 \in L^2(0,T;{L^2(\R^n))} $.
\end{lemma}

\begin{proof}
We show the boundedness of $\mathcal{P}$, estimating the $L^2$-norm of $\mathcal{P}$ as follows: Let $k$ be a function in $L^2(0,T;L^2(\R))$. Exploiting equation (\ref{estimateGamma}) and the assumptions stated above, we get

\begin{equation*}
\begin{split}
&\|\mathcal{P}(k)(\cdot,t)\| ^2_{L^2(\R^n )} =  \int_{\R^n} |\mathcal{P}(k)(x,t)|^2 \ \ dx\\
= & \int_{\R^n} \left| A_0(x)\exp\left( \frac{\int_{\R^n} \phi(k(y,t))\Gamma_{\mu} (x,y)\ dy}{\int_{\R^n} \phi(k(y,t))\Gamma_{\varepsilon} (x,y)\  dy + \xi} \ t\right)\ p(k(x,t)) \right|^2 \ dx\\
\le & \|A_0\|^2_{L^\infty(\R^n)}  \int_{\R^n} \exp\left(2t \frac{\int_{\R^n} \phi(k(y,t))\Gamma_{\mu} (x,y)\  dy}{\int_{\R^n}\phi( k(y,t))\Gamma_{\varepsilon} (x,y)\  dy + \xi} \right)\ p(k(x,t))^2   \ dx\\
%\le & \|A_0\|^2_{L^\infty(\R^n)}\int_0^T \int_{\R^n} \left|\exp\left(\frac{2t\varepsilon^n}{\mu^n}\right)p^2(k(x,t))\right| \ dxdt\\
\le & \|A_0\|^2_{L^\infty(\R^n)} \exp(\frac{2T\varepsilon^n}{\mu^n}) \int_{\R^n} (p(k(x,t)) - p(0))^2 \ dx\\
%=  & \|A_0\|^2_{L^\infty(\R^n)}  e^{\frac{2T\varepsilon^n}{\mu^n}} \int_0^T \int_{\R^n} |p(k(x,t))- p(0)|^2 \ dxdt\\
%\le & \|A_0\|^2_{L^\infty(\R^n)} e^{\frac{2T\varepsilon^n}{\mu^n}} L_p^2 \int_0^T \int_{\R^n} |k(x,t)|^2\ dxdt\\
\le &\|A_0\|^2_{L^\infty(\R^n)}   \exp(\frac{2T\varepsilon^n}{\mu^n}) L_p^2 \|k\|^2_{L^2(\R^n))} <\infty.
\end{split}
\end{equation*}

In order to prove the Lipschitz continuity of $\mathcal{P}$, we consider two functions $v_1,\ v_2\in L^2(0,T;L^2(\R^n))$ to estimate
%We estimate the difference
%\begin{align*}
%&\|\mathcal{P}(v_1) - \mathcal{P}(v_2)\|_{L^2(0,T;L^2(\R^n))}.
%\end{align*}% = \int_0^T \int_{R^n} |\mathcal{P}(v_1)(x,t) - \mathcal{P}(v_2)(x,t)|^2 dx dt
%Adding a `clever' zero, we can reformulate this as
\begin{align}
&\|\mathcal{P}(v_1)(\cdot,t) - \mathcal{P}(v_2)(\cdot,t)\|_{L^2(\R^n)} \leq \\[2mm]
%= & \|P(v_1)p(v_1) - P(v_1)p(v_2) + P(v_1)p(v_2) - P(v_2)p(v_2)\|_{L^2(0,T;L^2(\R^n))}\\[2mm]
& \|P(v_1)(\cdot,t)(p(v_1)(\cdot,t) - p(v_2)(\cdot,t))\|_{L^2(\R^n)} 
 + \|(P(v_1)(\cdot,t) - P(v_2)(\cdot,t))p(v_2)(\cdot,t)\|_{L^2(\R^n)}. \label{star5}
\end{align}
For the first term, we can deduce
\begin{align}
& \|P(v_1)(\cdot,t)(p(v_1)(\cdot,t) - p(v_2)(\cdot,t))\|_{L^2(\R^n)} 
\le & L_p\|P(v_1)(\cdot,t)\|_{L^\infty(\R^n)}\|v_1(\cdot,t)-v_2(\cdot,t)\|_{L^2(\R )}, \label{star6}
\end{align}
exploiting the Lipschitz continuity of the production function $p$ as claimed in Assumption \ref{AssP}. The productivity operator $P$ is bounded in $L^\infty(0,T,L^\infty(\R^n))$, which we get from
\begin{align}
&\|P(v)(\cdot,t)\|_{L^\infty(\R^n)}\\
:&= \mathrm{ess\ sup}_{x\in\R^n} \left| A_0(x) \exp\left(\frac{\int_{\R^n} \phi(v(y,t))\Gamma_\mu(x,y)dy}{\int_{\R^n} \phi(v(y,t))\Gamma_\varepsilon(x,y)dy+\xi}t\right)\right|\\
 &\le \|A_0\|_{L^\infty(\R^n)}\exp({\frac{T\varepsilon^n}{\mu^n}}) < \infty  \label{star7}
 %&=: \kappa_2 e^T <\infty,
\end{align}
again referring to Assumption \ref{AssP}. Finally we have for some constants $d_1, d_2$
\[ \|P(v_1)(\cdot,t)(p(v_1)(\cdot,t) - p(v_2)(\cdot,t))\|_{L^2(\R^n)}  \leq d_1 \exp(d_2 T)  \|v_1(\cdot,t)-v_2(\cdot,t)\|_{L^2(\R )} \]
For the estimation of the second term, we use the global boundedness of the production function $p$ and find
\begin{align*}
\|(P(v_1)(\cdot,t) - P(v_2)(\cdot,t))p(v_2)(\cdot,t)\|_{L^2(\R^n)}
%&\|p(v_2)(P(v_1)-P(v_2))\|_{L^2(0,T;L^2(\R^n))}\\[2mm]
%\le &\ \|p(v_2)\|_{L^\infty(0,T;L^\infty(\R^n))}\|P(v_1)-P(v_2)\|_{L^2(0,T;L^2(\R^n))}\\[2mm]
\le \ M_p \ \|P(v_1)(\cdot,t)-P(v_2)(\cdot,t)\|_{L^2(\R^n)}.
\end{align*}
We abbreviate the terms in $P$ as follows and define for $\nu\in\{\varepsilon,\mu\}$ 
\[\Phi_\nu(v)(x,t):= \int_{\R^n} \phi(v(y,t))\Gamma_\nu(x,y)dy.\]
Furthermore, note that
the boundedness of the fraction according to the inequality (\ref{estimateGamma}) yields
\[\left| \frac{\Phi_\mu(v_1)(x,t)}{\Phi_\varepsilon(v_1)(x,t) + \xi}t  -  \frac{\Phi_\mu(v_2)(x,t)}{\Phi_\varepsilon(v_2)(x,t)+ \xi}t  \right|\le 2\frac{\varepsilon^n}{\mu^n}\]
so that we can use the local Lipschitz-continuity of the exponential function denoting the Lipschitz constant  by $L_{exp}$..
For any $t\in(0,T)$ we have
\begin{equation*}
\begin{split}
&\|P(v_1)(\cdot,t)-P(v_2)(\cdot,t)\|_{L^2( \R^n)}  \\
%&  \left\| A_0(\cdot)\left[\exp\left( \frac{\int_{\R^n} \phi(v_1(y,s))\Gamma_\mu(\cdot,y)dy}{\int_{\R^n} \phi(v_1(y,s))\Gamma_\varepsilon(\cdot,y)dy + \xi}s \right) \right.\right.\\
%&\left.\left. \hspace*{5.5cm} - \exp\left( \frac{\int_{\R^n} \phi(v_2(y,s))\Gamma_\mu(\cdot,y)dy}{\int_{\R^n} \phi(v_2(y,s))\Gamma_\varepsilon(\cdot,y)dy + \xi}s \right)\right] \right\|_{L^2(\R^n)}.\\
%&\le\left\|A_0(\cdot)\left[\exp\left( \frac{\int_{\R^n} \phi(v_1(y,s))\Gamma_\mu(\cdot,y)dy}{\int_{\R^n} \phi(v_1(y,s))\Gamma_\varepsilon(\cdot,y)dy + \xi}s \right) - \exp\left( \frac{\int_{\R^n} \phi(v_2(y,s))\Gamma_\mu(\cdot,y)dy}{\int_{\R^n} \phi(v_2(y,s))\Gamma_\varepsilon(\cdot,y)dy + \xi}s \right)\right] \right\|_{L^2(\R^n)}.\\
%\end{split}
%\end{equation*}
%for all $v_1,v_2$, so we can use the local Lipschitz continuity of the exponential function with Lipschitz constant $L_{\exp}>0$ and get
%
%\begin{equation*}
%\begin{split}
= &\left\|A_0(\cdot)\left[\exp\left( \frac{\Phi_\mu(v_1)(\cdot,t)}{\Phi_\varepsilon(v_1)(\cdot,t) + \xi}t \right) - \exp\left( \frac{\Phi_\mu(v_2)(\cdot,t)}{\Phi_\varepsilon(v_2)(\cdot,t)+ \xi}t \right)\right] \right\|_{L^2(\R^n)}\\[2mm]
\le & \ t L_{exp} \ \left\|A_0(\cdot)\left[\frac{\Phi_\mu(v_1)(\cdot,t)}{\Phi_\varepsilon(v_1)(\cdot,t) + \xi} -  \frac{\Phi_\mu(v_2)(\cdot,t)}{\Phi_\varepsilon(v_2)(\cdot,t)+ \xi}\right]  \right\|_{L^2(\R^n)}.\\[2mm]
%= s L_{exp} & \left\|A_0(\cdot)\left[\frac{\Phi_\mu(v_1)(\cdot,s)(\Phi_\varepsilon(v_2)(\cdot,s)+ \xi)}{(\Phi_\varepsilon(v_1)(\cdot,s) + \xi)(\Phi_\varepsilon(v_2)(\cdot,s)+ \xi)} -  \frac{\Phi_\mu(v_2)(\cdot,s)(\Phi_\varepsilon(v_1)(\cdot,s) + \xi)}{(\Phi_\varepsilon(v_1)(\cdot,s) + \xi)(\Phi_\varepsilon(v_2)(\cdot,s)+ \xi)} \right] \right\|_{L^2(\R^n)}.\\
\end{split}
\end{equation*}
Omitting arguments, we can estimate
\begin{equation*}
\begin{split}
&\left| \frac{\Phi_\mu(v_1)(\cdot,t)}{\Phi_\varepsilon(v_1)(\cdot,t) + \xi} -  \frac{\Phi_\mu(v_2)(\cdot,t)}{\Phi_\varepsilon(v_2)(\cdot,t)+ \xi} \right| \\
%&\left|  \frac{\Phi_\mu(v_1)(\Phi_\varepsilon(v_2)+\xi) - \Phi_\mu(v_2)(\Phi_\varepsilon(v_1)+\xi)}{(\Phi_\varepsilon(v_1)+\xi)(\Phi_\varepsilon(v_2)+\xi) }\right|\\[2mm]
%=&\left| \frac{\Phi_\mu(v_1)\Phi_\varepsilon(v_2) + \Phi_\mu(v_1)\xi - \Phi_\mu(v_2)\Phi_\varepsilon(v_1) - \Phi_\mu(v_2)\xi}{(\Phi_\varepsilon(v_1)+\xi)(\Phi_\varepsilon(v_2)+\xi) }  \right|\\[2mm]
=& \left| \frac{\Phi_\mu(v_1)\Phi_\varepsilon(v_2) -\Phi_\mu(v_1)\Phi_\varepsilon(v_1) + \Phi_\mu(v_1)\Phi_\varepsilon(v_1)  - \Phi_\mu(v_2)\Phi_\varepsilon(v_1)+ \Phi_\mu(v_1)\xi - \Phi_\mu(v_2)\xi}{(\Phi_\varepsilon(v_1)+\xi)(\Phi_\varepsilon(v_2)+\xi) }  \right|\\[2mm]
%= & \left| \frac{\Phi_\mu(v_1)(\Phi_\varepsilon(v_2) -\Phi_\varepsilon(v_1)) + \Phi_\varepsilon(v_1)( \Phi_\mu(v_1) - \Phi_\mu(v_2))+\xi( \Phi_\mu(v_1) - \Phi_\mu(v_2))}{(\Phi_\varepsilon(v_1)+\xi)(\Phi_\varepsilon(v_2)+\xi) }  \right|\\[2mm]
\le & \left| \frac{\Phi_\mu(v_1)}{(\Phi_\varepsilon(v_1)+\xi)(\Phi_\varepsilon(v_2)+\xi)}\right| \left| \Phi_\varepsilon(v_1) -\Phi_\varepsilon(v_2) \right| \\
& \ +  \left(\left| \frac{\Phi_\varepsilon(v_1)}{(\Phi_\varepsilon(v_1)+\xi)(\Phi_\varepsilon(v_2)+\xi)}\right| + \left|\frac{\xi}{(\Phi_\varepsilon(v_1)+\xi)(\Phi_\varepsilon(v_2)+\xi) }\right|\right) \left| \Phi_\mu(v_1) - \Phi_\mu(v_2) \right|\\[2mm]
\le & \frac{\varepsilon^n}{\mu^n\xi} |\Phi_\varepsilon(v_1) -\Phi_\varepsilon(v_2) | + \frac{2}{\xi} |\Phi_\mu(v_1) - \Phi_\mu(v_2) |,
\end{split}
\end{equation*}
where we have estimated the term
\begin{equation*}
\begin{split}
\left| \frac{\Phi_\mu(v_1)}{(\Phi_\varepsilon(v_1)+\xi)(\Phi_\varepsilon(v_2)+\xi)}\right|\le \frac{1}{\xi} \left| \frac{\Phi_\mu(v_1)}{(\Phi_\varepsilon(v_1)+\xi)}\right|\le \frac{1}{\xi}\left(\frac{\varepsilon}{\mu}\right)^n,
\end{split}
\end{equation*}
applying the inequality (\ref{estimateGamma}) and $(\Phi_\varepsilon(v_1)+\xi)\Phi_\varepsilon(v_2)\ge 0$ by assumption. The other terms are estimated in a similar way.\\
Therefore,
\begin{equation*}
\begin{split}
& \left\|A_0(\cdot)\left[\frac{\Phi_\mu(v_1)(\cdot,s)}{\Phi_\varepsilon(v_1)(\cdot,s) + \xi} -  \frac{\Phi_\mu(v_2)(\cdot,s)}{\Phi_\varepsilon(v_2)(\cdot,s)+ \xi}\right]  \right\|_{L^2(\R^n)}\\[2mm]
\le &\ \frac{2}{\xi} \left\|A_0(\cdot)\left[\Phi_\mu(v_1)(\cdot,s)-\Phi_\mu(v_2)(\cdot,s)\right]\right\|_{L^2(\R^n)} + \frac{\varepsilon^n}{\xi\mu^n} \left\|A_0(\cdot)\left[\Phi_\varepsilon(v_1)(\cdot,s)-\Phi_\varepsilon(v_2)(\cdot,s)\right]\right\|_{L^2(\R^n)}.
\end{split}
\end{equation*}
We now apply the Lipschitz continuity of the function $\phi$ and get
\begin{equation*}
\begin{split}
&\left\|A_0(\cdot)\left[\Phi_\nu(v_1)(\cdot,s)-\Phi_\nu(v_2)(\cdot,s)\right] \right\|_{L^2(\R^n)}\\[2mm]
%= & \left\|A_0(\cdot)\left[\int_{\R^n} \phi(v_1(y,s))\Gamma_\nu(\cdot,y)dy -  \int_{\R^n} \phi(v_2(y,s))\Gamma_\nu(\cdot,y)dy\right] \right\|_{L^2(\R^n )}\\[2mm]
=&\left( \int_{\R^n}\left| A_0(x)  \int_{R^n} (\phi(v_1(y,s))- \phi(v_2(y,s)))\Gamma_\nu(x,y)dy\right|^2 dx\right)^{\frac{1}{2}}\\[2mm]
\le & \left( \int_{\R^n} \left( \int_{\R^n} (\phi(v_1(y,s))-\phi(v_2(y,s)))^2dy\right) \left( \int_{\R^n}A_0(x)^2\Gamma_\nu^2(x,y)dy\right) dx\right)^{\frac{1}{2}}\\[2mm]
&\le L_\phi \|v_1-v_2\|_{L^2(\R^n)} \left( \int_{\R^n}\int_{\R^n} A_0^2(x)\Gamma_\nu^2(x,y)dydx\right)^{\frac{1}{2}},
\end{split}
\end{equation*}
for $\nu\in\{\varepsilon,\mu\}$. \\

Note that, since the kernel function $\Gamma_\nu$ is a multivariate Gaussian probability density function, we obtain
\[ess\sup_{x\in\R^n}\int_{\R^n}\frac{1}{(2\pi\nu^2)^n}\exp\left(-\frac{\|x-y\|_2^2}{\nu^2}\right) dy = \frac{1}{(2\nu\sqrt{\pi})^n}<\infty.\]
Since we have assumed
\[\|A_0\|_{L^2(\R^n)}<\infty,\]
we can finally deduce
\begin{equation*}
\begin{split}
\|P(v_1)(\cdot,s)-P(v_2)(\cdot,s)\|_{L^2( \R^n)} \le s \kappa_3\|v_1(\cdot,s)-v_2(\cdot,s)\|_{L^2(\R^n)},
\end{split}
\end{equation*}
with a positive constant $\kappa_3$.\\
We have from (\ref{star5}) using (\ref{star6}) and (\ref{star7})
\begin{equation*}
\begin{split}
\| \mathcal{P}(k_1)(\cdot,t)  - \mathcal{P}(k_2)(\cdot,t) \|_{L^2(\R^n)} \leq c_2 \|k_1 - k_2\|_{L^2(\R^n)} \quad \mbox{for all}~~ t \in [0,T]
%\\
%\|\mathcal{P}(v_1)&-\mathcal{P}(v_2)\|_{L^2(0,T;L^2(\R^n))}\\
%&\le  \max\{\kappa_2,\kappa_3\} (T+e^T)\ \|v_1(\cdot,s)-v_2(\cdot,s)\|_{L^2(0,T;L^2(\R^n))} 
\end{split}
\end{equation*}
\end{proof}

\begin{theorem}\label{Existence_Ramsey_unbounded}
Let $k_0\in L^2(\R^n ),\ c \in L^2(0,T;H^{-1}(\R^n))$ and let the functions $p$, $\phi$, and $A_0$ satisfy Assumption \ref{AssP}. Then the capital accumulation equation in the nonlocal spatial Ramsey model with endogenous productivity growth (\ref{semilineq}) admits a unique weak solution $k\in W (0,T)$.
\end{theorem}

\begin{proof}
We give the proof to Theorem \ref{Existence_Ramsey_unbounded}, following a common technique which is based on Banach's fixed point theorem and the Lipschitz continuity of the nonlinearity, which we have already shown in Lemma \ref{regularityP}.\\
First, we fix $T^*\in(0,T)$. We show that the solution mapping, which maps a right-hand side to the solution of the linearized differential equation, is a contraction for $T^*$ sufficiently small. \\
Let $v\in  \mathcal{C}([0,T^*];L^2(\R^n))$, for short $\mathcal{C}(0,T^*;L^2(\R^n))$. As proven in Lemma \ref{regularityP}, $\mathcal{P}(v)\in L^2(0,T^*;L^2(\R^n ))$, where we have used the inequality 
\[\|v\|_{L^2(0,T^*;L^2(\R^n))}\le \|v\|_{L^\infty(0,T^*;L^2(\R^n))}\]
for all finite $T^*$. According to Theorem 2 by \citet[p. 513]{dautray} there exists a unique weak solution $u\in W(0,T^*)$ of 
\[u_t  -\mathcal{L}(u)+\delta u = \mathcal{P}(v) - f\ \mbox{\hspace{1cm} in } \R^n\times(0,T^*),\]
with $u(\cdot,0)=k_0(\cdot)$ on $\R^n$. The embedding 
\[W(0,T^*)\hookrightarrow \mathcal{C}(0,T^*; L^2(\R^n))\]
guarantees that $u$ is an element of $\mathcal{C}(0,T^*;L^2(\R^n ))$ (cf. \citet[p. 521]{dautray}). This defines the operator 
\[S: \mathcal{C}(0,T^*;L^2(\R^n ))\to \mathcal{C}(0,T^*;L^2(\R^n)),\ S(v)=u.\] 
In the following, we prove that $S$ is a contraction. Consider the difference \\ $S(v_1)-S(v_2)$ for two arbitrary functions $v_1,v_2\in  \mathcal{C}(0,T^*;L^2(\R^n ))$ with $S(v_1)=u_1$ and $S(v_2)=u_2$. We define the function $u:=u_1-u_2\in W(0,T^*)$ and deduce from the weak formulation with $u$ as the variational variable
\begin{equation*}
\begin{split}
&\int_0^t \int_{\R^n} u_t(x,s)u(x,s)\ dx + {\bf a} (u,u)(s)\ ds=\\
&\int_0^t\int_{\R^n} (\mathcal{P}(v_1)(x,s)-\mathcal{P}(v_2)(x,s))u(x,s)\   dxds,
\end{split}
\end{equation*}
for all $t\in[0,T^*]$.\\ 
We can estimate the left-hand side (LHS) using a calculation from Lemma \ref{PropyA} (i):
\begin{equation*}
\begin{split}
LHS&\ge \int_0^t\int_{\R^n}u_t(x,s)u(x,s)\  dxds+ \alpha \int_0^t\int_{\R} |\nabla_x u|^2 (x,s)\  \  dxds\\
&+ \delta \int_0^t \int_{\R^n} u^2(x,s)\  dxds - \kappa_1 \int_0^t \|\nabla_x u(s)\|_{L^2(\R^n )}\|u(s)\|_{L^2(\R^n )}ds,\\[4mm]
\end{split}
\end{equation*}
where 
\[0\le \kappa_1 :=  \beta \int_{\R^n} |z| \frac{1}{\sqrt{(2\pi \varepsilon^2)^n}}\exp\left(\frac{-\|z\|_2^2}{2\varepsilon^2} \right)\ dz .\]
Applying Young's inequality for an arbitrary $\eta>0$, we get
\begin{equation*}
\begin{split}
LHS & \ge \int_0^t\int_{\R^n}u_t(x,s)u(x,s)\  dxds + \alpha \int_0^t\int_{\R} |\nabla_xu|^2 (x,s)\  \  dxds\\
&+ \delta \int_0^t \int_{\R^n} u^2(x,s)\  \ dxds  -  \int_0^t \frac{\eta}{2}\|\nabla_xu(s)\|^2_{L^2(\R^n)}+ \frac{\kappa_1^2}{2\eta}\|u(s)\|^2_{L^2(\R^n)}\ ds. \\
\end{split}
\end{equation*}
We choose $\eta\le 2\alpha$, which yields together with the identity
\[\int_0^t\int_{\R^n}u_t(x,s)u(x,s)\  dxds = \frac{1}{2} \|u(\cdot,t)\|^2_{L^2(\R^n)}\]
the following estimate for the left-hand side
\begin{equation*}
\begin{split}
 LHS  &\ge \frac{1}{2} \|u(\cdot,t)\|^2_{L^2(\R^n)} + \int_0^t \delta\|u(\cdot,s)\|^2_{L^2(\R^n)} + \frac{\kappa_1^2}{2\eta}
\|u(\cdot,s)\|^2_{L^2(\R^n)}\ ds .
\end{split}
\end{equation*}
For the right-hand side (RHS), we get
\begin{equation*}
\begin{split}
&\int_0^t\int_{\R^n} (\mathcal{P}(v_1)(x,s)-\mathcal{P}(v_2)(x,s))(u(x,s))\ dxds\\
%\le & \int_0^t  e^s \kappa_2 \|v_1(\cdot,s)-v_2(\cdot,s)\|_{L^2(\R^n)}\|u(\cdot,s)\|_{L^2(\R^n)}ds\\
&  + \int_0^t s \kappa_3 \|v_1(\cdot,s)-v_2(\cdot,s)\|_{L^2(\R^n)} \|u(\cdot,s)\|_{L^2(\R^n)} ds\\
\le & \max\{\kappa_2,\kappa_3\} \int_0^t (s+e^s)  \|v_1(\cdot,s)-v_2(\cdot,s)\|_{L^2(\R^n)} \|u(\cdot,s)\|_{L^2(\R^n)} ds
\end{split}
\end{equation*}
according to Lemma \ref{regularityP}.\\

Applying Young's inequality for a $\varsigma>0$ and denoting by $\kappa_\infty:=\max\{\kappa_2,\kappa_3\}$ yields
\begin{equation*}
\begin{split}
& \kappa_\infty \int_0^t (s+e^s)  \|v_1(\cdot,s)-v_2(\cdot,s)\|_{L^2(\R^n)} \|u(\cdot,s)\|_{L^2(\R^n)} ds\\
\le & \frac{\kappa_\infty^2}{2\varsigma}\int_0^t (s+e^s)^2 \|v_1(\cdot,s)-v_2(\cdot,s)\|_{L^2(\R^n)}^2ds + \frac{\varsigma}{2} \int_0^t  \|u(\cdot,s)\|^2_{L^2(\R^n)} ds.
\end{split}
\end{equation*}

All in all, we have
\begin{equation*}
\begin{split}
& \frac{1}{2} \|u(\cdot,t)\|^2_{L^2(\R^n)} +  \int_0^t \delta\|u(\cdot,s)\|^2_{L^2(\R^n)} + \frac{\kappa_1^2}{2\eta}
\|u(\cdot,s)\|^2_{L^2(\R^n)}\ ds \\
&\le \frac{ \kappa_\infty^2}{2\varsigma}\int_0^t (s+e^s)^2 \|v_1(\cdot,s)-v_2(\cdot,s)\|_{L^2(\R^n)}^2ds + \frac{\varsigma}{2} \int_0^t  \|u(\cdot,s)\|^2_{L^2(\R^n)} ds.
\end{split}
\end{equation*}

Sorting the inequality leads to 
\begin{equation*}
\begin{split}
\frac{1}{2} \|u(\cdot,t)\|^2_{L^2(\R^n)} \le & \left( \frac{\varsigma}{2}  - \frac{\kappa_1^2}{2\eta} - \delta\right) \int_0^t \|u(\cdot,s)\|^2_{L^(\R^n)}ds \\
&\  +\frac{ \kappa_\infty^2}{2\varsigma}\int_0^t (s+e^s)^2 \|v_1(\cdot,s)-v_2(\cdot,s)\|_{L^2(\R^n)}^2ds,
\end{split}
\end{equation*}
where we can choose the parameters $\varsigma$ and $\eta$, such that $(\varsigma/2-\kappa_1^2/2\eta - \delta)\ge 0$. Taking the maximum for all $t\in[0,T^*]$, we end up with
\begin{equation*}
\begin{split}& \frac{1}{2} \|u_1-u_2\|_{L^{\infty}(0,T^*;L^2( \R^n))}^2 \\
& \le T^* \left( \frac{\varsigma}{2} - \frac{\kappa_1^2}{2\eta} - \delta\right) \|u_1-u_2\|^2_{L^{\infty}(0,T^*;L^2( \R^n))}  + \frac{ \kappa_\infty^2}{2\varsigma}\int_0^{T^*}(s+e^s)^2ds\|v_1-v_2\|^2_{L^{\infty}(0,T^*;L^2( \R^n))}
\end{split} 
\end{equation*}
as $\delta,\kappa^2\ge 0$. Hence, we have
 \begin{equation*}
 \begin{split}
\|u_1-u_2\|^2_{L^{\infty}(0,T^*;L^2(\R^n))} \le C(T^*)\ \|v_1-v_2\|^2_{L^{\infty}(0,T^*;L^2(\R^n))}, \hspace*{1cm} (\#)
 \end{split}
 \end{equation*}
 with
 \[C(T^*):=\frac{\hat{\kappa}_\infty(\frac{T^{*^3}}{3} + 2e^{T^*}T^*-2e^{T^*}+2+\frac{1}{2}e^{2T^*}-\frac{1}{2})}{\frac{1}{2}- T^*\left( \frac{\varsigma}{2} - \frac{\kappa_1^2}{2\eta} - \delta\right) }.\]
Taking the limit $T^*\to 0$ yields
\[C(T^*)\to 0.\] Especially, the exists a $T^*\in\R_+$, such that $C(T^*)<1$. \\

Note that we can divide by $\frac{1}{2}- T^*\left( \frac{\varsigma}{2} - \frac{\kappa_1^2}{2\eta} - \delta\right)$, because we can choose $\varsigma,\eta$ appropriately, such that the term is positive, at least for small $T^*$. So, all in all, we have shown that $S$ is a contraction for $T^*$ small enough and we can apply the fixed point theorem of Banach which yields the existence of a unique fixed point $S(u)=u$ on $W(0,T^*)$.  Now, we have to construct a solution $k$ on the whole time-space-cylinder, but since the local solution $u$ is independent of the time horizon $T^*$, we can proceed on the interval $[T^*,2T^*]$ using the same arguments as above with a new initial condition $u(\cdot,T^*)$. After finitely many steps, we can construct a weak solution $k\in W(0,T)$ of (\ref{semilineq}). Moreover, this solution is unique, which follows from the inequality $(\#)$.
\end{proof}

\begin{lemma}\label{a_priori_L2}
Let $c\in L^2(0,T;L^2(\R^n))$. If the bilinear form ${\bf a} $ is coercive, then there exist two constants $\mathrm{C}_1, \mathrm{C}_2> 0$ such that the solution $k\in W (0,T)$ of (\ref{semilineq}) satisfies the following a priori estimate:
\begin{equation}\label{apriori_semilinear}
\|k\|_{L^2(0,T;H^1(\R^n))} \ +\ \mathrm{C}_1\|k\|_{L^\infty(0,T;L^2(\R^n))} \le \mathrm{C}_2 \left( \|k_0\|_{L^2(\R^n)} + \|c\|_{L^2(0,T;L^2(\R^n))} + 1 \right).
\end{equation}
\end{lemma}

\begin{proof}
The coercivity assumption on ${\bf a} $ yields a constant $c_{coer}>0$ such that \\ ${\bf a} (k,k)\ge c_{coer} \|k\|^2_{H^1(\R^n)}$ for all $k\in {H^1(\R^n)}$. Now, we fix a $t\in[0,T]$ and derive the weak formulation of (\ref{semilineq}) for the test function $k\in W (0,T)$. We then have
\begin{equation*}
\begin{split}
\int_0^t \int_{\R^n} \frac{\partial k}{\partial s} k\  \ dxds + \int_0^t {\bf a} (k(s),k(s))\ ds = \int_0^t \int_{\R^n} \mathcal{P}(k) k\ dxds - \int_0^t \int_{\R^n} c\ k\  dxds .
\end{split}
\end{equation*}
To estimate the right-hand side, we exploit the Lipschitz continuity of $p$, $p(0)=0$, and the boundedness of the fractional term in the productivity growth operator (\ref{estimateGamma}). Altogether, this yields
\begin{equation*}
\begin{split}
RHS & \le \int_0^t \|\mathcal{P}(k)\|_{L^2(\R^n)}\|k\|_{L^2(\R^n)} ds + \int_0^t \|c\|_{L^2(\R^n)}\|k\|_{L^2(\R^n)}ds\\
&\le \|A_0\|_{L^\infty(\R^n)}  \int_0^t \left(\int_{\R^n}( e^{\frac{s\varepsilon^n}{\mu^n}} p(k))^2  dx \right)^{\frac{1}{2}} \left( \int_{\R^n} |k|^2  dx \right)^{\frac{1}{2}} ds+ \int_0^t \|c\|_{L^2(\R^n)}\|k\|_{L^2(\R^n)}ds\\
&\le \|A_0\|_{L^\infty(\R^n)} L_p  \int_0^t e^{\frac{s\varepsilon^n}{\mu^n}} \|k\|^2_{L^2(\R^n)} ds +  \int_0^t \|c\|_{L^2(\R^n)}\|k\|_{L^2(\R^n)}ds.\\
\end{split}
\end{equation*}
In order to estimate the left-hand side, we use the coercivity assumption for ${\bf a} $,
\begin{equation*}
\begin{split}
LHS& = \int_0^t \int_{\R^n} \frac{\partial k}{\partial s} k  \ dxds + \int_0^t {\bf a}  (k,k) ds\\
&\ge \int_0^t \int_{\R^n} \frac{\partial k}{\partial s} k  \ dxds + c_{coer} \int_0^t \|k\|^2_{H^1(\R^n)}ds\\
&= \frac{1}{2} \|k(t)\|^2_{L^2(\R^n)} - \frac{1}{2} \|k_0\|^2_{L^2(\R^n)} + c_{coer} \int_0^t \|k\|^2_{H^1(\R^n)}ds.
\end{split}
\end{equation*}

Combining both estimates and applying Young's inequality with $\eta_1,\eta_2>0$, and taking the maximum of all $t\in[0,T]$, we have 
%\begin{equation*}
%\begin{split}
%& \frac{1}{2} \|k(t)\|^2_{L^2(\R^n)} + c_{coer} \int_0^t \|k\|^2_{H^1(\R^n)}ds \\
%\le &  \|A_0\|_{L^\infty(\R^n)} L_p \left( \frac{\eta_1\mu^n}{4\varepsilon^n}\left(e^{\frac{2\varepsilon^n t}{\mu^n}}-1\right) + \frac{1}{2\eta_1} \int_0^t \|k\|_{L^2(\R^n)}^2 ds\right)  \\
%& + \frac{\eta_2}{2}   \int_0^t \|c\|_{L^2(\R^n)}^2 ds + \frac{1}{2\eta_2} \int_0^t \|k\|^2_{L^2(\R^n)}ds + \frac{1}{2} \|k_0\|^2_{L^2(\R^n)}.
%\end{split}
%\end{equation*}

%Taking the maximum of all $t\in[0,T]$, we have
\begin{equation*}
\begin{split}
& \frac{1}{2} \|k\|^2_{L^\infty(0,T;L^2(\R^n))}+ c_{coer} \|k\|^2_{L^2(0,T;H^1(\R^n))} \\
\le &  \|A_0\|_{L^\infty(\R^n)} L_p \left( \frac{\eta_1\mu^n}{4\varepsilon^n}\left(e^{\frac{2T\varepsilon^n}{\mu^n}}-1\right) + \frac{1}{2\eta_1}  \|k\|_{L^2(0,T;H^1(\R^n))}^2 \right)  \\
& + \frac{\eta_2}{2}   \|c\|_{L^2(0,T;L^2(\R^n))}^2  + \frac{1}{2\eta_2}\|k\|^2_{L^2(0,T;H^1(\R^n))} + \frac{1}{2} \|k_0\|^2_{L^2(\R^n)}.
\end{split}
\end{equation*}

We multiply with $2$, sort all terms, and end up with
\begin{equation*}
\begin{split}
& \|k\|^2_{L^\infty(0,T;L^2(\R^n))} + \left(2c_{coer} - \frac{\|A_0\|_{L^\infty(\R^n)} L_p}{\eta_1} - \frac{1}{\eta_2}\right)\|k\|^2_{L^2(0,T;H^1(\R^n))}\\
&\le  \eta_1\ \|A_0\|_{L^\infty(\R^n)} L_p  \frac{\mu^n}{2\varepsilon^n}\left(e^{\frac{2\varepsilon^n \ T}{\mu^n}}-1\right)  + \eta_2 \|c\|_{L^2(0,T;L^2(\R^n))}^2 +  \|k_0\|^2_{L^2(\R^n)}.
\end{split}
\end{equation*}

Since $\eta_1,\eta_2>0$ were arbitrary, we choose both constants large enough such that
 \[\left(2c_{coer} - \frac{\|A_0\|_{L^\infty(\R^n)} L_p}{\eta_1} - \frac{1}{\eta_2}\right)> 0,\]
which completes the proof.\\
\end{proof}

The following a priori estimate of the weak solution is crucial for the proof of existence of an optimal control.

\begin{lemma}\label{Ta_priori_W0T}
There exists a constant $\tilde{C}>0$ such that weak solution of (\ref{semilineq}) satisfies
\begin{equation}\label{a_priori_W0T}
\|k\|_{W (0,T)} \le \tilde{C}(\|c\|_{L^2(0,T;L^2(\R^n))} + \|k_0\|_{L^2(\R^n)} + 1).
\end{equation}
\end{lemma}

\begin{proof}
We follow the proof of Theorem 3.13 by \citet[p.121]{troltzsch}. First note that
\[\|k\|_{W(0,T)}^2 = \|k\|^2_{L^2(0,T;H^1(\R^n))} + \|k_t\|^2_{L^2(0,T;H^{-1}(\R^n))}. \]
For the first term, we have already proven in Lemma \ref{a_priori_L2} that there exists a constant $\mathrm{C}>0$ such that the inequality
\[\|k\|^2_{L^2(0,T;H^1(\R^n))} \le \mathrm{C} \left( \|k_0\|_{L^2(\R^n)} + \|c\|_{L^2(0,T;L^2(\R^n))} + 1 \right)^2\]
holds true. In order to estimate the second term, a bit more work has to be done. First, we define the linear functionals $F_i(t):H^1(\R^n)\to \R$, $i=1,...,5$ as
\begin{equation*}
\begin{split}
& F_1(t):\ v \mapsto \langle \alpha \nabla_x k(t),\nabla_x v \rangle _{L^2(\R^n)},\\
& F_2(t):\ v \mapsto \langle \beta \int_{\R^n}(k(y,t)-k(\cdot,t))\Gamma_\varepsilon(\cdot,y)dy,v \rangle _{L^2(\R^n)},\\
& F_3(t):\ v \mapsto \langle \delta k(t),v  \rangle _{L^2(\R^n)},\\
& F_4(t):\ v \mapsto \langle \mathcal{P}(k)(t),v \rangle _{L^2(\R^n)},\\
& F_5(t):\ v \mapsto \langle c(t),v \rangle _{L^2(\R^n)}.\\
\end{split}
\end{equation*}
The weak formulation of the PIDE then yields
\[\|k_t\|_{L^2(0,T;H^{-1}(\R^n))} \le \sum_{i=1}^5 \|F_i\|_{L^2(0,T;H^{-1}(\R^n))}.\]
We estimate all summands separately. It holds
\begin{equation*}
\begin{split}
|F_1(t)v| & = |\alpha\langle \nabla_x k(t),\nabla_xv \rangle_{L^2(\R^n)}| \le \alpha \|\nabla_x k(t)\|_{L^2(\R^n)} \|\nabla_x v\|_{L^2(\R^n)} \\[2mm]
&\le \alpha  \|k(t)\|_{H^1(\R^n)} \|v\|_{H^1(\R^n)},
\end{split}
\end{equation*}
hence
\begin{equation*}
\begin{split}
\|F_1\|^2_{L^2(0,T;H^{-1}(\R^n))} & \le \int_0^T \|F_1(t)\|^2_{H^{-1}(\R^n)} dt
\le \int_0^T \hat{c} \|k(t)\|_{H^1(\R^n)}^2dt 
 = \hat{c}\|k\|^2_{L^2(0,T;H^1(\R^n))}.
\end{split}
\end{equation*}
for a constant $\hat{c}\ge 0$. Finally, from Lemma \ref{a_priori_L2} we know that
\begin{equation*}
\begin{split}
\|F_1\|^2_{L^2(0,T;H^{-1}(\R^n))} & \le \hat{c}\|k\|^2_{L^2(0,T;H^1(\R^n))}
\le \mathrm{C}\left(\|c\|_{L^2(0,T;L^2(\R^n))} + \|k_0\|_{L^2(\R^n)} + 1\right)^2.
\end{split}
\end{equation*}
with $\mathrm{C}> 0$. For the other summands we proceed analogously and deduce
\begin{equation*}
\begin{split}
|F_2(t)v| & \le \beta \kappa  \|k(t)\|_{H^1(\R^n)} \|v\|_{H^1(\R^n)} ,\\[3mm]
|F_3(t)v|&  \le \delta \|k(t)\|_{H^1(\R^n)} \|v\|_{H^1(\R^n)} ,\\[3mm]
|F_4(t)v| &\le \|\mathcal{P}(k)(t)\|_{L^2(\R^n)} \|v\|_{L^2(\R^n)} \le \mathrm{\hat{c}}(t) \|k(t)\|_{H^1(\R^n)} \|v\|_{H^1(\R^n)} ,
\end{split}
\end{equation*}
for contants which we have already derived in the proof of Lemma \ref{a_priori_L2}. For the last term, it holds
\[|F_5(t)v|\le \|c(t)\|_{L^2(\R^n)} \|v\|_{H^1(\R^n)}.\]
Taking the maximum of all $t\in[0,T]$ and combining all estimates for $F_i,\ i=1,...,5$ we get
\[\|k_t\|^2_{L^2(0,T;H^1(\R^n))} \le \hat{\mathrm{C}} \left( \|c\|_{L^2(0,T;L^2(\R^n))} + \|k_0\|_{L^2(\R^n)} + 1\right)^2,\]
and together with the a priori estimate in Lemma \ref{a_priori_L2}, we finally have
\begin{equation*}
\begin{split}
\|k\|^2_{W(0,T)} & = \|k\|^2_{L^2(0,T;H^1(\R^n))} + \|k_t\|^2_{L^2(0,T;H^{-1}(\R^n))} \\
&\le\tilde{\mathrm{C}} \left( \|c\|_{L^2(0,T;L^2(\R^n))} + \|k_0\|_{L^2(\R^n)} + 1\right)^2 .
\end{split}
\end{equation*}

\end{proof}

In the one dimensional case, we are even able to show the essential boundedness and continuity of the weak solution of the capital accumulation equation, raising a result by \cite{lady} from compact subsets of the domain of interest to the whole, untruncated spatial domain:

\begin{lemma}\label{kCont}
We consider the case $n=1$. Let $L_p$ be the Lipschitz constant of the production function $p$ and let the assumptions $(1)-(4)$ hold. Assume that the initial value function $k_0\in L^2(\R)$ is also H\"older continuous of exponent $\lambda>0$ on $\R$. Let $\mathcal{U}_{ad}$ and $k_0$ be chosen such that there exists a constant $\theta>0$ satisfying
\begin{equation}\label{k0Uad}
\begin{split}
(i)\ \ & 4L_p^2\le \theta \mbox{ and } \\
(ii) \ & |k_0(x)| + \int_0^T |c(x,s)|^2ds  < \frac{1}{16L_pTe^{\theta T/2}} \quad \forall c\in\mathcal{U}_{ad}.
\end{split}
\end{equation} 
Then, the weak solution of the spatial nonlocal Ramsey model (\ref{state_control}) is bounded and continuous on $\R\times[0,T]$ for every $c\in\mathcal{U}_{ad}$.
\end{lemma}

\begin{proof}
The boundedness of the weak solution is a direct application of Theorem 3.1 by \citet[p.956]{Ran}. We can interpret the nonlocal diffusion term as compact perturbation of the right-hand side and use the boundedness of the productivity operator $\mathcal{P}$. By assumption, $k_0$ and $c\in\mathcal{U}_{ad}$ are chosen such that (\ref{k0Uad}) is satisfied. Hence, the theorem mentioned above yields that there exists a constant $0<M(c):=M<\infty$, such that $\|k\|_{L^\infty(\R\times [0,T])}\le M$.\\

In order to prove the continuity of the weak solution on $\R\times[0,T]$, a bit more work has to be done. We want to apply a result from \citet{lady}, where the authors prove the local H\"older continuity of any essentially bounded weak solution of a quasi-linear differential equation of parabolic type on $\Omega\times [0,T]$. Although \citet{lady} assumed a bounded and open domain $\Omega\subset\R^n$, we can adapt the statement to our case. Since the result does not require any boundary conditions on $\partial \Omega$, we can apply Theorem 1.1 in \citet[p.419]{lady} also to the spatially unbounded case and get the H\"older continuity of exponent $\lambda_1>0$ of the weak solution on all compact subsets $K\times[t_1,t_2]\Subset\R\times(0,T)$ with $K\subset \R$ compact and $0<t_1<t_2<T$. The H\"older constant $\lambda_1$  depends only on $M$ and the coercivity constant of the bilinear form (\ref{a}), $c_{coer}>0$.\\
 Since the initial value function $k_0$ is assumed to be H\"older continuous of exponent $\lambda$, we can apply the extended result of Theorem 1.1 in \citet[p.419]{lady} and get the local H\"older continuity of exponent $\lambda_2>0$ depending on $M,\ c_{coer}$ and $\lambda$ of the weak solution also in $t=0$. Remark that we can extend our left-hand sides $c$ and $\mathcal{P}$ to $[0,T+\varepsilon]$ for any $\varepsilon>0$, so we can assume the local H\"older continuity of the weak solution also in $T$ without any loss of generality. Note that we have all assumptions in \citet[p.418]{lady} fulfilled since the solution $k$ is bounded and the constants $\alpha$ and $\beta$ in (\ref{L}) are positive. Hence, we have
\[k \in \mathcal{C}^{\lambda',\lambda'/2}(K\times[0,T]) \quad \forall K\Subset \R,\]
with some positive $\lambda'$ depending only on $M,\ T,\ c_{coer}$, and $\lambda$. Every H\"older continuous function is also continuous, hence we have
\[k \in \mathcal{C}(K\times[0,T]) \quad \forall K\Subset \R.\]

Considering the exhaustion by compact sets of $\R$, meaning a sequence of compact sets $\{K_m\}_{m\in\N}$ with $K_m\subset \mathring{K}_{m+1}$ and $\bigcup_{m\in\N} K_m=\R$, we can extend the local result to the whole unbounded $\R$ due to the locality of continuity. This means, for every $x_0\in\R$, there exists a $m$ large enough such that $(x_0,t)\in \mathring{K}_m\times[0,T]$, where $\mathring{K}$ denotes the interior of a set $K$.\\
\end{proof}

\section{The Ramsey Equilibrium - Existence of an Optimal Control}

The capital accumulation equation, which we studied intensively in the previous section, is crucial for the capital and consumption distribution in the considered economy. The Ramsey equilibrium in the nonlocal setting of our case is given as the solution of the following optimal control problem:\\

\begin{equation}\label{objective_control}
\begin{split}
\min_{k,c} \quad & \int_0^T \int_{\R} - U(c(x,t))e^{-\tau t - \gamma |x|^2} \ dxdt \\
 &+ \frac{1}{2\rho_1} \|k(\cdot,T) - k_T(\cdot)\|^2_{L^2(\R)}+ \frac{1}{2\rho_2 }\|\min \{0, k\}\|^2_{L^2(0,T;L^2({\R}))},
\end{split}
\end{equation}

\begin{equation}\label{state_control}
\begin{split}
\hspace*{-2cm}\mbox{ s.t.}\quad  \frac{\partial k}{\partial t} - \mathcal{L}(k) + \delta k - \mathcal{P}(k) &=  -c \qquad \text{ on } \R\times(0,T),\\
k(x,0)&=k_0(x) \quad \text{ on } {\R},\\
\ c \  &\in  \ \mathcal{U}_{ad}.\\
\end{split}
\end{equation}

The optimal control and the corresponding state variable define the market equilibrium according to the second welfare theorem of economics. \\

The set $\mathcal{U}_{ad}$ denotes the set of feasible controls. The first assumption, which we need to make in order to prove the existence of an optimal control in the nonlocal spatial Ramsey model, states:

\begin{itemize}
\item[(1)] {\em  $\mathcal{U}_{ad}$ is a bounded, closed, and convex subset of  $L^2(0,T;L^2(\R))$.}
\end{itemize}
For example, for a given maximal consumption function $c_{max}\in L^2(\R)\cap L^\infty(\R)$ and a maximum aggregated consumption level $\overline{C}\in\R_+$, we consider 
\[\mathcal{U}_{ad}:=\{c\in L^2(0,T;L^2(\R)):\ 0\le c \le c_{max} \ \wedge \ \|c\|_{L^2(0,T;L^2(\R))}\le \overline{C} \}.\]

The assumptions concerning the Lipschitz continuity and the uniform boundedness of the production function $p$ and the essential boundedness of the initial productivity distribution $A_0$ remain valid.

\begin{itemize}
\item[(2)] {\em The production function $p$ and the initial productivity distribution $A_0$ satisfy Assumption \ref{AssP}. As in (\ref{L}), the diffusion operator is defined as
\[\mathcal{L}(k)(x,t):= \alpha \Delta k(x,t) + \beta \int_{\R} (k(y,t)-k(x,t))\Gamma_\varepsilon(x,y)dy ,\]
for some weights $\alpha,\beta>0$ and $\varepsilon>0$ and the productivity-production operator is given as in (\ref{gamma}) by
\[\mathcal{P}(k)(x,t) := A_0(x) \exp\left( \frac{\int_{\R} \phi(k(y,t))\Gamma_\mu(x,y)dy}{\int_{\R} \phi(k(y,t))\Gamma_\varepsilon(x,y)dy}\ t \right)\ p(k(x,t)).\]}
\end{itemize}

Here, the necessary properties of $\phi$ motivate the next assumption:

\begin{itemize}
\item[(3)] {\em The nominal function $\phi$ is (Lipschitz) continuous. We assume that $\phi(k)>0$ for all $k$. }
\end{itemize}
In the following we assume that $\phi$ denotes a continuous approximation of the absolute value function. The example that we have in mind is
\[\phi(k) = \sqrt{k^2 + \eta},\]
which depends on a parameter $\eta$. We assume this parameter to be a priori defined and fixed, so it is convenient to omit the dependence of $\phi=\phi(\eta)$ on this parameter. However, for the proof of the existence of an optimal control, any continuous and positive function $\phi$ is sufficient.\\

Considering the objective function, we assume the following.

\begin{itemize}
\item[(4)] {\em The utility function $U:\R\to\R$ is bounded and locally Lipschitz continuous, hence there exists a constant $K$, such that 
\[|U(0)|\le K,\]
and a constant $L(M)$, such that for all $c_1,c_2\in[-M,M]$
\[|U(c_1)-U(c_2)|\le L(M)|c_1-c_2|.\]
Moreover, we assume that $U$ is concave.}
\end{itemize}

It is worth to mention once more, that $U$ is the utility function which describes the consuming sector in the Ramsey economy. Hence, the concavity is essential for the economic interpretation. The assumptions on $U$, together with the measurability of the function $(x,t)\mapsto e^{-\tau t-\gamma\|x\|_2^2}$ for $\tau,\gamma>0$, are necessary in order to guarantee that the objective function $\mathcal{J}$ is convex, continuous, and bounded from below on $\mathcal{U}_{ad}$.\\

\begin{lemma}\label{regKappa}
The nonlocal diffusion operator
\[\mathcal{NL}(\cdot)(x,t):\ k \mapsto \int_\R (k(y,t)-k(x,t))\Gamma_\varepsilon(x,y) dy\] is continuous from $L^2(0,T;L^2(\R))$ to $L^2(0,T; L^2(\R))$.
\end{lemma}

\begin{proof}
We proceed as in Lemma \ref{regularityP}. Since $L^2(0,T;L^2(\R))$ is a vector space, we can estimate the two terms separately. 
To estimate the first term, we apply Young's inequality for convolution and recall that the kernel function $\Gamma_\varepsilon$ indeed has the form $\Gamma_\varepsilon(x,y) = \Gamma_\varepsilon(x-y)$. Since \[\int_\R  \Gamma_\varepsilon(x)dx=1,\] we can deduce
\begin{equation*}
\begin{split}
\left\|\int_\R k(y,t)\Gamma_\varepsilon(x-y)dy\right\|^2_{L^2(0,T;L^2(\R))} 
&\ =\int_0^T \|k(\cdot,t)*\Gamma_\varepsilon(\cdot-y)\|_{L^2(\R)} dt\\
&\ \le \int_0^T \|k(\cdot,t)\|_{L^2(\R)}\|\Gamma_\varepsilon\|_{L^1(\R)} dt\\
&\ =\|k\|^2_{L^2(0,T;L^2(\R))},
\end{split}
\end{equation*}

The estimation of the second term yields
\begin{equation*}
\begin{split}
\left\|\int_\R k(x,t)\Gamma_\varepsilon(x,y)dy\right\|^2_{L^2(0,T;L^2(\R))}  
& =\int_0^T \int_\R \left( \int_\R k(x,t)\Gamma_\varepsilon(x,y)dy\right)^2 dx dt\\
&= \int_0^T \int_\R k(x,t)^2 \left(\int_\R \Gamma_\varepsilon(x,y)dy\right)^2 dx dt\\
& \le \|k\|_{L^2(0,T;L^2(\R))}^2,
\end{split}
\end{equation*}
again exploiting $\int_\R \Gamma_\varepsilon(x,y)dy = 1$ for all $x \in\R$ by definition.

\end{proof}

\begin{theorem}\label{ExistenceOCP}
Let the assumptions $(1)-(4)$ hold. Moreover, assume that $k_0 \in L^2(\R)\cap\mathcal{C}^{\lambda,\lambda/2}(\R)$ for some $\lambda>0$ and $\mathcal{U}_{ad}$ fulfill the assumptions of Lemma \ref{kCont}, hence are given such that the weak solution of the capital accumulation equation is bounded and continuous. Then, there exist an optimal control $\overline{c}\in \mathcal{U}_{ad}$ and a corresponding optimal state $\overline{k}\in W(0,T)$ of the spatial nonlocal Ramsey model (\ref{objective_control}) and (\ref{state_control}).
\end{theorem}

\begin{proof}
As already shown in Theorem \ref{Existence_Ramsey_unbounded}, the state equation has a unique weak solution $k:=k(c) \in W(0,T)$ for every control $c\in\mathcal{U}_{ad}$ and $k_0\in L^2(\R)$. The additional assumptions on the initial value function $k_0$ then yield, according to Lemma \ref{kCont}, the uniform boundedness of this weak solution, i.e. the existence of a constant $0<M:=M(c)<\infty$ such that
\[\|k\|_{L^\infty(\R\times [0,T])}\le M,\]
for all states corresponding to a control $c\in\mathcal{U}_{ad}$ and the continuity of $k$, i.e. \\$k\in C(\R\times[0,T])$.\\

Due to the boundedness of $\mathcal{U}_{ad}$, the uniform boundedness of $k$ in $W(0,T)$ according to Lemma \ref{Ta_priori_W0T}, and the assumption on the objective, there exists a finite infimum $J_{inf}$ of $\mathcal{J}$. Since $L^2(0,T;L^2({\R}))$ is reflexive, we can choose a minimizing sequence $(c_m)_{m\in\N}$ in $\mathcal{U}_{ad}$ that has a weak convergent subsequence $(c_{m_j})_{j\in\N}$ with limit $\overline{c}\in L^2(0,T;L^2(\R))$. Without any loss of generality, we can identify this subsequence with $(c_m)_{m\in\N}$. Assumption $(1)$ states that $\mathcal{U}_{ad}$ is closed and convex, thus weakly sequentially closed, which guaranties that $\overline{c}\in\mathcal{U}_{ad}$. Hence,  we get
\[c_m \rightharpoonup \overline{c}\in \mathcal{U}_{ad},\ m\to \infty.\]
This sequence of controls defines a sequence of corresponding states \\$(k_m)_{m\in\N}:=(k(c_m))_{m\in\N}$. We define
\[\rho_m := \mathcal{P}(k_m)\quad  \mbox{  and  } \quad \kappa_m := \int_{\R} (k_m(y,t) - k_m(x,t))\Gamma_\varepsilon(x,y) dy.\]
As already shown in the Lemmas \ref{regularityP} and \ref{regKappa}, $\rho_m$ and $\kappa_m$ are elements of \\$L^2(0,T;L^2(\R ))$ for $k_m\in L^2(0,T;L^2(\R))$. Moreover, the continuity of $\mathcal{P}$ and $\mathcal{NL}$ in $L^2(0,T;L^2(\R))$ yields the uniform boundedness of the sequences due to the uniform boundedness of $(k_m)_{m\in\N}$ in $W(0,T)$. Here, we have applied Lemma \ref{Ta_priori_W0T} and recalled that $\mathcal{U}_{ad}$ is assumed to be bounded in $L^2(0,T;L^2(\R))$. Hence, we can assume that there exist some subsequences, again denoted by $(\rho_m)_{m\in\N}$ and $(\kappa_m)_{m\in\N}$ without any loss of generality, that converge weakly to some $\overline{\rho}$ and $\overline{\kappa}$ in $L^2(0,T;L^2(\R))$.\\

Now, we consider the linear parabolic initial value problems given by
\begin{equation*}
\begin{split}
\frac{\partial k_m}{\partial t} - \alpha\Delta k_m + \delta k_m & = p_m+ \kappa_m - c_m \quad \mbox{ on } \R \times(0,T)\\
k_m(\cdot,0) & = k_0 \hspace*{1.96cm} \mbox{  on } \R 
\end{split}
\end{equation*}
for $m\in\N$. We know that the right-hand side converges weakly towards $\overline{p}+\overline{\kappa} - \overline{c}$ in $L^2(0,T;L^2(\R))$, hence also in $L^2(0,T;H^{-1}(\R))$ since the embedding $L^2(0,T;L^2(\R)) \hookrightarrow L^2(0,T;H^{-1}(\R)) $ is continuous. Due to the continuity of the solution mapping, that maps a right-hand side and an initial value function to the solution of a linear parabolic differential equation (c. \citealp[p.382]{wloka}), this mapping is also weakly continuous from $L^2(0,T; H^{-1}(\R))\times L^2(\R)$ to $W(0,T)$. Thus we get the weak convergence of the left-hand side as well. We have
\[k_m\rightharpoonup \overline{k} \mbox{  in  }W(0,T),\ m\to\infty. \]
Moreover, the continuity of the solution mapping guarantees that $\overline{k}\in W(0,T)$.\\
With the same arguments as used in the proof of Lemma \ref{kCont}, we get
\[k_m\rightharpoonup \overline{k} \mbox{  in  } C^{\lambda',\lambda'/2}(K\times[0,T]),\ m\to\infty,\]
for all compact subsets $K$ of $\R $ and an $\lambda'>0$, depending on $M,\ T,\ coer$, and $\lambda$. \\
It is true that
\[C^{\lambda',\lambda'/2}( K\times [0,T]) \hookrightarrow^c C(K\times [0,T]),  \]
(cf. \citealp[p.12]{adams}), thus we get the strong convergence of the sequence of states in the space of continuous functions on all compact subsets $K\times[0,T]$ of $\R\times [0,T]$.\\

Now, we need to show the convergence of the integrals in the weak formulation. So far, we have derived
\begin{equation}\label{cmkm}
\begin{split}
(i)\ \ \ & c_m \rightharpoonup \overline{c} \mbox{  in  } L^2(0,T;L^2(\R )), \\
(ii)\ \ & k_m\rightharpoonup \overline{k} \mbox{  in  } W(0,T),\\
(iii)\ & k_m\to\overline{k} \mbox{  in  } C(K\times[0,T]) \mbox{  for all  } K\Subset \R .
\end{split}
\end{equation}

Due to the a priori estimates in (\ref{apriori_semilinear}) and (\ref{a_priori_W0T}) and the weak (or weak star) compactness of unit balls in the spaces $L^2(0,T;H^1(\R))$, $L^\infty(0,T;L^2(\R))$, and $L^2(0,T;H^{-1}(\R))$, we can adapt the arguments by \citet[p.515]{dautray} and extract a subsequence $(k_{\overline{m}})_{\overline{m}\in\N}$ with
\begin{equation}\label{km_weakL2}
\begin{split}
(i)\ \ & k_{\overline{m}} \rightharpoonup \overline{k} \mbox{ in } L^2(0,T;H^1(\R)),\\
(ii)\ & k_{\overline{m}}\rightharpoonup^* \overline{k} \mbox{ in } L^\infty(0,T;L^2(\R)).
\end{split}
\end{equation}
Note that the embedding $W(0,T)\hookrightarrow L^2(0,T;H^1(\R))$ is continuous, hence weakly continuous, which guarantees that the subsequence $(k_{\overline{m}})_{\overline{m}\in\N}$ converges to the same limit in $L^2(0,T;H^1(\R))$ as $(k_m)_{m\in\N}$ in $W(0,T)$.\\

For now, we choose the test functions $\psi:=\varphi \otimes v$ for $\varphi\in \mathcal{C}_0^\infty([0,T[)$ with $\varphi(0)\neq 0$ and $v\in \mathcal{C}^\infty_0(\R)$. We derive the weak formulation of the capital accumulation equation as
\begin{equation*}
\begin{split}
&-\int_0^T \langle k_{\overline{m}}(t),\psi_t(t)\rangle_{L^2(\R)} dt + \int_0^T {\bf a} (k_{\overline{m}}(t),\psi(t))\ dt -\int_0^T \langle \mathcal{P}(k_{\overline{m}})(t),\psi(t)\rangle_{L^2(\R)}dt\\
&=-\int_0^T\langle c_{\overline{m}}(t),\psi(t)\rangle_{L^2(\R)}dt + \langle k_0,\psi(0)\rangle_{L^2(\R)}, \\
&\psi=\varphi\otimes v,\ \forall \varphi \in \mathcal{C}^\infty_0([0,T[),\ \varphi(0)\neq 0,\ v\in\mathcal{C}^\infty_0(\R),
\end{split}
\end{equation*}
with the definition of the bilinear form in \ref{a}. \\
From (\ref{cmkm})$(i)$, we can deduce
\[\int_0^T \langle c_{\overline{m}}(t),\psi(t)\rangle_{L^2(\R)} dt \to \int_0^T \langle \overline{c}(t),v\rangle_{L^2(\R)} \varphi(t)\ dt,\ m\to\infty\]
and from (\ref{km_weakL2})$(i)$, we get

\[\int_0^T \langle k_{\overline{m}}(t),\psi'_{\overline{m}}(t) \rangle_{L^2(\R)} dt \to \int_0^T \langle \overline{k}(t),\psi'(t) \rangle_{L^2(\R)} dt,\ m\to\infty.\]
We can rewrite the bilinear form ${\bf a}(k,\psi)$ in verctorial form as $\langle A k,\psi\rangle_{H^1(\R)}$ with $Ak(\cdot)\in L^2(0,T;H^{-1}(\R))$ (cf. \citealp[p.515]{dautray}), hence (\ref{km_weakL2})$(i)$ implies
\[\int_0^T {\bf a} (k_{\overline{m}}(t),\psi(t))dt \to \int_0^T {\bf a} (\overline{k}(t),\psi(t))dt\quad \mbox{ for  }{\overline{m}}\to\infty.\]

So far, we were able to use the same arguments as \citet[p.515]{dautray}. The convergence of the nonlinear productivity term needs some further analysis. We exploit the strong convergence of the sequence of states on compact sets, (\ref{cmkm})$(iii)$ and the properties of the kernel function $\Gamma_\varepsilon$, respectively $\Gamma_\mu$. We start with the exponential term. Due to the boundedness of every $k_{\overline{m}}$ in $L^\infty(\R\times[0,T])$, the continuity of $\phi$, and the property of $\Gamma_\varepsilon$ to be decreasing for large absolute values of input variables, we can choose a radius $R>0$ large enough such that
\[\int_{\R \backslash \mathcal{B}_R(0)} \phi(k_{\overline{m}}(y,t))\Gamma_\varepsilon(x,y)dy \le \tilde{\varepsilon}/4\]
for all ${\overline{m}}\in\N$ and an $\tilde{\varepsilon}>0$. 
Then it is also 
\[\int_{\R \backslash \mathcal{B}_R(0)} \phi(k_{\overline{m}}(y,t))\Gamma_\mu(x,y)dy \le \tilde{\varepsilon}/4\]
for $\mu<\varepsilon$. With this, we get
\begin{equation*}
\begin{split}
& \left| \int_{\R } \phi(k_{\overline{m}}(y,t))\Gamma_\varepsilon(x,y)dy - \int_{\R } \phi(\overline{k}(y,t))\Gamma_\varepsilon(x,y)dy  \right|\\
\le & \left| \int_{\mathcal{B}_R(0)} (\phi(k_{\overline{m}}(y,t))- \phi(\overline{k}(y,t)))\Gamma_\varepsilon(x,y)dy  \right| \\
& +  \left|\int_{\R \backslash \mathcal{B}_R(0)} \phi(k_{\overline{m}}(y,t))\Gamma_\varepsilon(x,y)dy \right| +  \left|\int_{\R \backslash \mathcal{B}_R(0)} \phi(\overline{k}(y,t))\Gamma_\varepsilon(x,y)dy \right|\\
\le &  \left| \int_{\mathcal{B}_R(0)} (\phi(k_{\overline{m}}(y,t))- \phi(\overline{k}(y,t)))\Gamma_\varepsilon(x,y)dy  \right|  + \tilde{\varepsilon}/2.
\end{split}
\end{equation*}
By assumption, $\phi$ is continuous, hence there exists a $N\in\N$ such that if $|k_{\overline{m}}-\overline{k}|\le \delta$ for ${\overline{m}}\ge N$, it is
$|\phi(k_{\overline{m}}(y,t)) - \phi(\overline{k}(y,t))|\le \tilde{\varepsilon}(N)$. This yields
\[ \int_{\mathcal{B}_R(0)}|(\phi(k_{\overline{m}}(y,t))- \phi(\overline{k}(y,t)))|\Gamma_\varepsilon(x,y)dy \le \tilde{\varepsilon}(N) \int_{\mathcal{B}_R(0)} \Gamma_{\varepsilon}(x,y)dy.\]
We can choose $N$ large enough such that $\tilde{\varepsilon}(N)\le \tilde{\varepsilon}/2$ and end up with
\[\left| \int_{\R } \phi(k_{\overline{m}}(y,t))\Gamma_\varepsilon(x,y)dy - \int_{\R } \phi(\overline{k}(y,t))\Gamma_\varepsilon(x,y)dy  \right| \le \tilde{\varepsilon},\]
since $\int_{\mathcal{B}_R(0)} \Gamma_{\varepsilon}(x,y)dy \le 1$.\\
According to (\ref{estimateGamma}), the exponential term is bounded by
\[\exp\left( \frac{\int_\R \phi(k_{\overline{m}}(y,t))\Gamma_\mu(x,y)dy}{\int_\R \phi(k_{\overline{m}}(y,t))\Gamma_\varepsilon(x,y)dy} t\right)\le e^{\frac{T\varepsilon}{\mu}} \]
for all $k_{\overline{m}}$. Hence, we can exploit the property of the chosen test function $v$ having compact support on $\R$. We can finally show the convergence
\[ \int_0^T \langle \mathcal{P}(k_{\overline{m}})(t),\psi(t) \rangle _{L^2(\R)}dt \to \int_0^T \langle \mathcal{P}(\overline{k})(t),\psi(t) \rangle_{L^2(\R)}dt,\ \overline{m}\to\infty,\]
since $k_{\overline{m}}\to \overline{k}$ strongly on all compact sets $K\times[0,T]$. Combining all limits for ${\overline{m}}\to\infty$ in the weak formulation,we obtain
\begin{equation*}
\begin{split}
&-\int_0^T \langle \overline{k}(t),v\rangle_{L^2(\R)}\varphi'(t) dt + \int_0^T {\bf a} (\overline{k}(t),v)\varphi(t) dt -\int_0^T \langle \mathcal{P}(\overline{k})(t),v\rangle_{L^2(\R)}\varphi(t) dt\\
&=\langle k_0,\psi(0)\rangle_{L^2(\R)}-\langle \overline{c}(t),v\rangle_{L^2(\R)}\varphi(t) dt.
\end{split}
\end{equation*}

Since this equality has to hold for every $\varphi \in \mathcal{C}^\infty_0([0,T[)$ with $\varphi(0)\neq 0$ and $v\in \mathcal{C}_0^\infty(\R)$, we have finally shown that $\overline{k}\in W(0,T)$ is indeed a weak solution of the capital accumulation equation in the nonlocal spatial Ramsey model with endogenous productivity growth.\\

The only thing left to show in order to finish this proof of existence of an optimal control is the optimality of $(\overline{c},\overline{k})$. But this follows immediately from the convexity and continuity of the objective function: Recall that every continuous and convex function is lower semicontinous. Hence, for 
\begin{equation*}
\begin{split}
J(c,k) &=  \int_0^T \int_{\R} - U(c(x,t))e^{-\tau t - \gamma x^2} \ dxdt \\
&+ \frac{1}{2\rho_1 }\| k(\cdot,T) - k_T(\cdot)\|^2_{L^2(\R)} + \frac{1}{2\rho_2 }\|\min \{0, k\}\|^2_{L^2(0,T;L^2({\R}))}\\
&:= F(c) + Q(k)
\end{split}
\end{equation*}
it follows
\begin{equation*}
\begin{split}
J_{inf}&=\lim_{\overline{m}\to\infty} J(c_{\overline{m}},k_{\overline{m}}) = \lim_{\overline{m}\to\infty} F(c_{\overline{m}})+ \lim_{\overline{m}\to\infty} Q(k_{\overline{m}})\ge F(\overline{c}) + Q(\overline{k}) = J(\overline{c},\overline{k}).
\end{split}
\end{equation*}
Since $J_{inf}$ was the infimum of $\mathcal{J}$, we get the equality.
\end{proof}

The proof of existence of an optimal control is crucial, not only for the mathematical study. From an economic point of view, this means that there exists a competitive market equilibrium in the closed spatialized Ramsey economy, where households may be heterogeneous in their initial capital distribution, productivity is heterogeneous in space and time, and no interaction with the surrounding takes place. Due to its complexness, the nonlocal spatial Ramsey model with endogenous productivity growth is quite general. \cite{brito01,brito04,brito12}, \cite{boucekkine,boucekkine13}, and \cite{camacho} admit that their spatial versions of the Ramsey model are not well-posed in the sense of Hadamard, at least if they consider a quite general, convex utility function and no further restrictions on the set of interest. As already mentioned, all approaches analyzing the (local) model with respect to existence of an optimal control are based on the theory of classical solutions. We considered a weaker notion of solution and were able to proof not only the existence of a weak solution of the capital accumulation equation but also the existence of an optimal control. Since our model is very general, we can capture the dynamics of the local model by Brito as a special case (for example by setting $\beta=0$, $\mu=\varepsilon$). Hence, we have enhanced the economic theory on the spatial Ramsey model and provided not only a proof of existence for our model, but also for the common spatial Ramsey models.

\bibliography{literature}

\begin{thebibliography}{}

\bibitem[Acemoglu, 2009]{acemoglu}
Acemoglu, D. (2009).
\newblock {\em Introduction to Modern Economic Growth}.
\newblock Princton University Press.

\bibitem[Adams and Fournier, 2003]{adams}
Adams, R. and Fournier, J. (2003).
\newblock {\em Sobolev Spaces}.
\newblock Elsevier, 2 edition.

\bibitem[Boucekkine et~al., 2013]{boucekkine13}
Boucekkine, R., Camacho, C., and Fabbri, G. (2013).
\newblock Spatial dynamics and convergence: The spatial {AK} model.
\newblock {\em Journal of Economic Theory}, 148:2719--2736.

\bibitem[Boucekkine et~al., 2009]{boucekkine}
Boucekkine, R., Camacho, C., and Zou, B. (2009).
\newblock Bridging the gap between growth theory and the new economic
  geography: The spatial {R}amsey model.
\newblock {\em Macroeconomic Dynamics}, 13:20--45.

\bibitem[Brito, 2001]{brito01}
Brito, P. (2001).
\newblock A {B}entham-{R}amsey model for spatially heterogeneous growth.
\newblock {\em Working Papers Depeartment of Economics}.
\newblock ISEG, University of Lisboa.

\bibitem[Brito, 2004]{brito04}
Brito, P. (2004).
\newblock The dynamics of growth and distribution in a spatially heterogeneous
  world.
\newblock {\em Working Papers Departement of Economics}.
\newblock ISEG, University of Lisboa.

\bibitem[Brito, 2012]{brito12}
Brito, P. (2012).
\newblock Global endogeneous growth and distributional dynamics.
\newblock {\em Munich Personal RePEc Archive}.

\bibitem[Camacho et~al., 2008]{camacho}
Camacho, C., Zou, B., and Briani, M. (2008).
\newblock On the dynamics of capital accumulatin across space.
\newblock {\em European Journall of Operational Research}, 186(2):451--465.

\bibitem[Dautray and Lions, 1992]{dautray}
Dautray, R. and Lions, J.-L. (1992).
\newblock {\em Mathematical Analysis and Numerical Methods for Science and
  Technology: Volume 5 Evolution Problems I}.
\newblock Springer-Verlag Berlin Heidelberg.

\bibitem[Lady\v{z}enskaya et~al., 1968]{lady}
Lady\v{z}enskaya, O., Solonnikov, V., and Ural\'{c}eva, N. (1968).
\newblock {\em Linear and Quasilinear Equations of Parabolic Type}, volume~23.
\newblock american mathematical society.
\newblock Translation of Mathematical Monographs.

\bibitem[Ramsey, 1928]{ramsey}
Ramsey, F. (1928).
\newblock A mathematical theory of saving.
\newblock {\em The Economic Journal}, 38(152):543--559.

\bibitem[Ran and Zhang, 2010]{Ran}
Ran, Q. and Zhang, T. (2010).
\newblock Existence and uniqueness of bounded weak solutions of a semilinear
  parabolic {PDE}.
\newblock {\em Journal of Theoretical Probability}, 23(4):951--971.

\bibitem[Tr{\"o}ltzsch, 2005]{troltzsch}
Tr{\"o}ltzsch, F. (2005).
\newblock {\em Optimale Steuerung partieller Differentialgleichungen: Theorie,
  Verfahren und Anwendungen}.
\newblock Vieweg+Teubner Verlag.

\bibitem[Wloka, 1982]{wloka}
Wloka, J. (1982).
\newblock {\em Partielle Differentialgleichungen: Sobolevr{\"a}ume und
  Randwertaufgaben}.
\newblock Mathematische Leitf{\"a}den. Teubner.

\end{thebibliography}

\end{document}